\documentclass[12pt]{amsart}

\newcommand{\BR}{\mathbb{{R}}}

\newtheorem{theorem}{Theorem}[section]
\newtheorem{proposition}[theorem]{Proposition}
\newtheorem{corollary}[theorem]{Corollary}
\theoremstyle{remark}
\newtheorem{remark}[theorem]{Remark}
\newtheorem{lemma}[theorem]{Lemma}

\theoremstyle{definition}

\def \e{\varepsilon}

\def \a{\alpha}

\def \D{\Delta}

\def \b{\beta}

\numberwithin{equation}{section}

\begin{document}

\author{Yu. Kolomoitsev and E. Liflyand}

\title[Multiple Fourier integrals]{Sufficient conditions for absolute convergence of multiple Fourier
integrals}



\address{Inst. Appl. Math. Mech., Rosa Luxemburg St. 74, Donetsk
83114, Ukraine}
\email{kolomus1@mail.ru}
\address{Department of Mathematics, Bar-Ilan University, 52900 Ramat-Gan, Israel}
\email{liflyand@math.biu.ac.il}

\begin{abstract}
Various new sufficient conditions for representation of a function
of several variables as an absolutely convergent Fourier integral
are obtained in the paper. The results are given in terms of $L^p$
integrability of the function and its partial derivatives, each with
the corresponding $p$. These $p$ are subject to certain relations
known earlier only for some particular cases. Sharpness and
applications of the obtained results are also discussed.

\end{abstract}

\maketitle

\section{Introduction}

If

\begin{eqnarray*}
f(y)=\int\limits_{\BR^d}g(x)e^{i(x,y)}dx,\qquad g\in L_1(\BR^d),
\end{eqnarray*}
we write $f\in A(\BR^d),$ with $\|f\|_A=\|g\|_{L_1(\BR^d)}.$

The possibility to represent a function as an absolutely convergent
Fourier integral has been studied by many mathematicians and is of
importance in various problems of analysis. For example, belonging
of a function $m(x)$ to $A(\BR^d)$ makes it to be an $L_1\to L_1$
Fourier multiplier (or, equivalently, $L_\infty\to L_\infty$ Fourier
multiplier); written $m\in M_1$ ($m\in M_\infty$, respectively). One
of such $m$-s attracted much attention in 50-80s (see, e.g.,
\cite{Wai},\cite{Fef},\cite[Ch.4, 7.4]{Stein}, \cite{Miy}, and
references therein):

\begin{eqnarray}\label{must}m(x):=m_{\alpha,\beta}(x)=\theta(x)\frac{e^{i|x|^\alpha}}
{|x|^\beta},                                     \end{eqnarray}
where $\theta$ is a $C^\infty$ function on $\mathbb R^d,$ which
vanishes near zero, and equals $1$ outside a bounded set, and
$\alpha, \beta>0.$ In is known that for $d\ge2$:

{\bf I)} If $\frac{\beta}{\alpha}>\frac{d}{2},$ then $m\in M_1
(M_\infty).$

{\bf II)} If $\frac{\beta}{\alpha}\le\frac{d}{2},$ then $m\not\in
M_1 (M_\infty).$

\noindent The first assertion holds true for $d=1$ as well, while
the second one only when $\alpha\ne1;$ however, the case
$\alpha=d=1$ is obvious.

Various sufficient conditions for absolute convergence of Fourier
integrals were obtained by Titchmarsh, Beurling, Karleman, Sz.-Nagy,
Stein, and many others. One can find more or less comprehensive and
very useful survey on this problem in \cite{SK}. Let us mention also
\cite{Peetre} and a couple of recent papers \cite{BDM, GW}.

New sufficient conditions of belonging to $A(\BR^d)$  are obtained
in this paper.

Let us unite certain of the known one-dimensional results closely
related to our study in the following theorem. First, it is natural
to consider functions $f\in A(\BR)$ that satisfy the condition

\medskip

{\bf (N-1)} {\it Let $f\in C_0(\BR),$ that is, $f\in C(\BR)$ and
$\lim f(t)=0$ as $|t|\to \infty$, and let $f$ be locally absolutely
continuous on} $\BR.$

\medskip

{\bf Theorem A-1.} {\it Let $f$ satisfy the condition {\bf (N-1)},
$f\in L^p(\BR)$ with $1\le p\le2,$ and $f'\in L^q(\BR)$ with
$1<q\le2.$ Then $f\in A(\BR).$}

\medskip

For the multivariate case, we need additional notations. Let
$\eta$ be $d$-dimensional vector with the entries either $0$ or
$1$ only. The inequality of vectors is meant coordinate wise. Here
and in what follows $D^\chi f$ for $\eta={\bf 0}=(0,0,...,0)$ or
$\eta={\bf 1}=(1,1,...,1)$ mean the function itself and the mixed
derivative in each variable, respectively, where

\begin{eqnarray*} D^\eta f(x)=\left(\prod\limits_{j:\, \eta_j=1}
\frac{\partial}{\partial x_j}\right)f(x).          \end{eqnarray*}

Let us give multidimensional results we are going, in a sense, to
generalize (see \cite{Mad} and \cite{Samko}, respectively).

\medskip

{\bf Theorem A1-d.} {\it Let $f\in L^2(\mathbb R^d).$ If all the
mixed derivatives (in the distributional sense)
$\frac{\partial^{\beta_j}} {\partial x_j^{\beta_j}} f(x)\in
L^2(\mathbb R^d),$ $j=1,2,...,d,$ where $\beta_j$ are positive
integers such that $\sum\limits_{j=1}^d\frac{1}{\beta_j}<2,$ then
$f\in A(\BR^d).$}

\medskip

{\bf Theorem A2-d.} {\it Let $f\in L^1(\mathbb R^d).$ If all the
mixed derivatives (in the distributional sense) $D^\eta f(x)\in
L^p(\mathbb R^d),$ $\eta\ne{\bf 0},$ where $1<p\le2,$ then $f\in
A(\BR^d).$}

\medskip

The outline of the paper is as follows. In the next section we
formulate the results. In Section \ref{aux} we present the needed
auxiliary results. Then, in Section \ref{od} we concentrate on the
one-dimensional version of our main results. In the last section we
give multidimensional proofs; one-dimensional arguments from the
preceding section will be intensively used.

We shall denote absolute positive constants by $C$, these constants
may be different in different occurrences.

\bigskip

\section{Main results}

It turns out that in several dimension there is a variety of results
in terms of different combinations of derivatives. It is still not
clear which one is "better", not always the sharpness of the
obtained results can be proved. We continue to study whether there
is a scale of such results, their sharpness and applicability.

Our first main result reads as follows.

\begin{theorem}\label{th1} Let $f\in C_0(\mathbb R^d)$ and let $f$
and its partial derivatives $D^\eta f,$ for all $\eta,$
$\eta\ne{\bf 1},$ be locally absolutely continuous on $(\mathbb
R\setminus \{0\})^d$ in each variable. Let $f\in L_{p_{\bf0}},$
$1\le p_{\bf0}<\infty,$ and let each partial derivative $D^\eta
f,$ $\eta\ne{\bf 0},$ belong to $L_{p_\eta}(\BR^d)$, where
$1<p_\eta<\infty$. If for all $\eta,$ $\eta\ne\bf0,$

\begin{eqnarray}\label{cond}\frac{1}{p_{\bf 0}}+\frac{1}
{p_\eta}>1, \end{eqnarray} then $f\in A(\mathbb R^d).$
\end{theorem}

\begin{remark}\label{rsh1} Condition (\ref{cond}) is sharp when
$\eta={\bf 1},$ while for other $\eta$ it is apparently not sharp.
\end{remark}

We can also obtain a result in which all the derivatives interplay
rather than the pairs $p_{\bf 0}$ and $p_\eta$.

\begin{theorem}\label{th2}
Let $f\in C_0(\mathbb R^d)$ and let $f$ and its partial
derivatives $D^\eta f,$ for all $\eta,$ $\eta\ne{\bf 1},$ be
locally absolutely continuous on $(\mathbb R\setminus \{0\})^d$ in
each variable. Let $f\in L_{p_{\bf0}},$ $1\le p_{\bf0}<\infty,$
and let each partial derivative $D^\eta f,$ $\eta\ne{\bf 0},$
belong to $L_{p_\eta}(\BR^d)$, where $1<p_\eta<\infty$. If
\begin{equation}\label{condPo}
\sum_{\bf{0\le \eta\le 1}}\frac{1}{p_\eta}>2^{d-1}
\end{equation}
and
\begin{equation}\label{condOtr1}
\sum_{\eta\neq {\bf 0}} \frac{1}{p_\eta}\le 2^{d-1},
\end{equation}
then $f\in A(\mathbb R^d)$.
\end{theorem}

This theorem can be given in the following equivalent form.

\medskip

\noindent\textbf{Theorem~\ref{th2}$'$.} {\it Let $f\in C_0(\mathbb
R^d)$ and let $f$ and its partial derivatives $D^\eta f,$ for all
$\eta,$ $\eta\ne{\bf 1},$ be locally absolutely continuous on
$(\mathbb R\setminus \{0\})^d$ in each variable. Let $f\in L_{p},$
$1\le p<p_{\bf0}\le\infty,$ and let each partial derivative
$D^\eta f,$ $\eta\ne{\bf 0},$ belong to $L_{p_\eta}(\BR^d)$, where
$1<p_\eta<\infty$. If
\begin{equation*}
\sum_{\bf{0\le \eta\le 1}}\frac{1}{p_\eta}=2^{d-1},
\end{equation*}
then $f\in A(\mathbb R^d)$.}

\medskip

\begin{remark}\label{dim2} We will see from the proofs of these
theorems that when $d=2$, the assertion holds true if we replace
assumption (\ref{condOtr1}) by $\frac{1}{p_0}+\frac{1}{p_{\bf1}}>1$.
\end{remark}

If to assume additionally that any of the $2^{d-1}-1$ derivatives
$D^\eta f$ are essentially bounded, then condition (\ref{condOtr1})
is satisfied. In this case the following statement holds.

\begin{corollary}\label{corBounDer}
Let $f\in C_0(\mathbb R^d)$ and let $f$ and its partial
derivatives $D^\eta f,$ for all $\eta,$ $\eta\ne{\bf 1},$ be
locally absolutely continuous on $(\mathbb R\setminus \{0\})^d$ in
each variable. Let $f\in L_{p_0}$, $1\le p_0<\infty$, and for the
derivatives $D^{\eta} f\in L_{p_\eta}$, $1<p_\eta<\infty$. Let
also $D^{\eta} f\in L_\infty$ for $|\eta|\le \frac{d}{2}$. If

\begin{equation}\label{condPoI*}
\sum_{\bf{0\le \eta\le 1}}\frac{1}{p_\eta}>2^{d-1},
\end{equation}
then $f\in A(\mathbb R^d)$.
\end{corollary}

For $d$ even, we can refine Corollary~\ref{corBounDer} as follows.

\begin{proposition}\label{Prop1}
Let $f\in C_0(\mathbb R^d)$, let $d$ be even, and let $f$ and its
partial derivatives $D^\eta f,$ for all $\eta,$ $\eta\ne{\bf 1},$
be locally absolutely continuous on $(\mathbb R\setminus \{0\})^d$
in each variable. Let $f\in L_{p_0}$, $1\le p_0<\infty$, and for
the derivatives $D^{\eta} f\in L_{p_\eta}$, $1<p_\eta<\infty$. Let
also $D^{\eta} f\in L_\infty$ for $|\eta|\le \frac{d}{2}-1$ and
$\frac{1}{p_{\bf0}}+\frac{1}{p_{\bf1}}>1$. If

\begin{equation*}\label{condPoI}
\sum_{\bf{0\le \eta\le 1}}\frac{1}{p_\eta}>2^{d-1},
\end{equation*}
then $f\in A(\mathbb R^d)$.
\end{proposition}

The next corollary gives conditions on which exponent decay of a
function $f$ and its derivatives ensures $f\in A(\mathbb R^d).$

\begin{corollary}\label{cor1} If

\begin{eqnarray}\label{gam}|D^\chi f(x)|\le
C\frac{1}{(1+|x|)^{\gamma_\chi}},\end{eqnarray} where
$\gamma_\chi>0$ for all $\chi,$ ${\bf 0}\le\chi\le{\bf 1},$ and

\begin{eqnarray}\label{coga}
\sum\limits_{{\bf 0}\le\chi\le{\bf 1}}
\gamma_\chi>d2^{d-1},
\end{eqnarray}
then $f\in A(\BR^d).$
\end{corollary}

It is often naturally to suppose that the derivatives of the same
order are of the same growth, for example, when the function is
radial, like $m_{\alpha,\beta}.$ The above result then reduces to
the next assertion.

\begin{corollary}\label{eqgr2}
Let $f\in C_0(\mathbb R^d)$ be a radial function, that is,
$f(x)=f_0(|x|)$, and let $f$ and its partial derivatives $D^\eta f,$
for all $\eta,$ $\eta\ne{\bf 1},$ be locally absolutely continuous
on $(\mathbb R\setminus \{0\})^d$ in each variable. Let $f\in
L_{p_0}$, $1\le p_0<\infty$, and for
$j=|\eta|=\eta_1+\dots+\eta_d>0$ the derivatives $D^{\eta} f\in
L_{p_j}$, $1<p_j<\infty$. Let also
\begin{equation}\label{condforRad}
f_0^{(s)}\in C(0,\infty),\quad \lim_{t\to\infty} t^s
f_0^{(s)}(t)=0,\quad 0\le s\le\frac{d-1}{2}.
\end{equation}
If
\begin{equation*}\label{condPoR}
\sum_{j=0}^d\binom{d}{j}\frac{1}{p_j}>2^{d-1},
\end{equation*}
and when $d$ is even
\begin{equation}\label{condPoR2}
\frac{1}{p_0}+\frac{1}{p_d}>1,
\end{equation} then $f\in
A(\mathbb R^d)$.
\end{corollary}

\begin{remark}
Note that  (\ref{condforRad}) are necessary conditions for belonging
to $A(\mathbb{R}^d)$ (see~\cite{Tiz} and \cite{GMJ}).
\end{remark}

\begin{remark}
We will see that Corollary~\ref{eqgr2} holds true if we replace
condition (\ref{condPoR2}) by
$$
\frac12\binom{d}{\frac d2}\frac{1}{p_{\frac d2}}+\sum_{j=\frac
d2+1}^d\binom{d}{j}\frac{1}{p_j}\le 2^{d-1}.
$$
\end{remark}

\begin{remark}\label{rsh2} We can prove that the conditions of the above results are
sharp only for certain $p_\eta.$ The point is that we make use of
$m_{\alpha,\beta}$ for which intermediate derivatives cannot be
arbitrary.
\end{remark}

As is mentioned, there is a variety of statements of above type. Let
us give one more, it can be proved similarly to those above.

\begin{theorem}\label{th2.13}
{\bf a)} \quad Let $f\in C_0(\mathbb R^d)\cap L_{p_0}(\mathbb
R^d)$, $1\le p_0<\infty$, $r>\frac d2$, $r\in \mathbb{N}$,
$\frac{\partial^{r-1}}{\partial x_j^{r-1}}f$ be locally absolutely
continuous in $x_j$, and $\frac{\partial^r}{\partial x_j^r}f\in
L_{p_j}(\mathbb R^d)$, $1<p_j<\infty$, $j=1,\dots,d$. If
$$
r<\frac{2r-d}{p_0}+\sum_{j=1}^d\frac{1}{p_j}\le \frac{2r-d}{p_0}+r,
$$
then $f\in A(\mathbb R^d)$.

{\bf b)}\quad Let $1\le p<\infty$, $1<q<\infty$, and
$$
\frac {2r-d}p+\frac {d}{q}<r.
$$
Then there is a function $f\in C_0(\mathbb R^d)\cap L_p(\mathbb
R^d)$ such that $\frac{\partial^r}{\partial x_j^r}f\in
L_{q}(\mathbb R^d)$, $j=1,\dots,d$, but $f\not\in A(\mathbb R^d)$.
\end{theorem}

 Theorem~\ref{th2.13} yields

\begin{corollary}
Let $r>\frac d2$, $r\in \mathbb{N}$, $\b,\a>0$, $\a\neq 1$, and
$\b>r(\a-1)$. If $\b>\frac{d\a}{2}$, then $m\in A(\mathbb{R}^d)$.
\end{corollary}

Here the point is that using other theorems results in a corollary
under more restrictive condition $\b>d(\a-1)$.

\section{Auxiliary results.}\label{aux}

One of the basic tools is the following lemma (see Lemma 4 in
\cite{Tiz} or Theorem 3 in \cite{Bes}, in any dimension).

\medskip

{\it Lemma B.}  { Let $f\in C_0(\mathbb R)$. If

\begin{eqnarray*}\sum\limits_{\nu=-\infty}^\infty 2^{\nu/2}
\left(\int\limits_{\mathbb R} |f(t+h(\nu))-f(t-h(\nu))|^2
dt\right)^{1/2}<\infty,                         \end{eqnarray*}
where $h(\nu)=\pi2^{-\nu},$ $\nu\in\mathbb Z,$ then $f\in A(\BR)$.}

\medskip

This lemma is a natural extension of the celebrated Bernstein's test
for the absolute convergence of Fourier series (see \cite[Ch.II, \S
6]{Kah}).

In order to formulate the multidimensional version, we denote

$$\Delta_u^{\eta,\,r} f(x)=\D_{u_1,\cdots
,u_d}^{\eta,\,r}f(x)=\prod\limits_{j\,:\,\eta_j=1}\D_{u_j}^{e_j,\,r}f(x),$$
where $\eta=(\eta_1,\dots,\eta_d)$ and $\D_{u_j}^{e_j,\,r}f$ is
defined as

\begin{eqnarray*}\label{in4}
\D_{u_j}^{e_j,\, r}f(x)=\sum_{k=0}^r \binom{r}{k}(-1)^k
f(x+(2k-r)u_j e_j), \quad 1\le j\le d.
\end{eqnarray*}
Here $e_j$ are basis unit vectors. Denote also $\Delta_u
f(x)=\D_{u_1,\cdots ,u_d}^{{\bf1},1}f(x)=\D_{u_1,\cdots ,u_d}f(x)$.

\medskip

{\it Lemma C.} { Let $f\in C_0(\BR^d)$. If

\begin{eqnarray*}
\sum_{s_1=-\infty}^\infty \cdots \sum_{s_d=-\infty}^\infty 2^ {\frac
12 \sum_{j=1}^d s_j}\|\D_{\frac{\pi}{2^{s_1}},\cdots,
\frac{\pi}{2^{s_d}}}(f)\|_2<\infty,
\end{eqnarray*}
where the norm is that in $L_2 (\BR^d),$ then} $f\in A(\BR^d)$.

\medskip

We will make use of the following Hardy type inequality (see
\cite[Cor.3.14]{KP}):

For $F\ge0$ and $1<q\le Q<\infty$

\begin{eqnarray}\label{harst}\qquad
\biggl(\int\limits_{\mathbb R}
\left[\,\int\limits_{t-h}^{t+h}F(s)\,ds\right]^Q
\,dt\biggr)^{1/Q}\le Ch^{1/Q+1/q'}\biggl(\int\limits_{\mathbb R}
F^q(t)\,dt\biggr)^{1/q}.
\end{eqnarray}
Here $\frac{1}{q}+\frac{1}{q'}=1.$ Similarly
$\frac{1}{p}+\frac{1}{p'}=1.$

We need the following direct multivariate generalization of
(\ref{harst}).

\begin{lemma}\label{sh}
For $F(u)\ge0$, $1\le k<d,$ and $1<q\le Q<\infty$

$$\biggl(\int\limits_{\mathbb R^d}
\left[\,\int\limits_{x_1-h_1}^{x_1+h_1}...\int\limits_{x_k-h_k}^{x_k+h_k}
F(u_1,...,u_k,x_{k+1},...,x_d)\,du_1...du_k\right]^Q
\,dx\biggr)^{1/Q}$$
\begin{eqnarray}\label{harstmu} \quad  \end{eqnarray}
$$\le C(h_1...h_k)^{\frac{1}{Q}+\frac{1}{q'}}\biggl(\int\limits_{\mathbb
R^{d-k}} \biggl[\int\limits_{\mathbb R^{k}}
F^q(x)\,dx_1...dx_k\biggr]^{Q/q} dx_{k+1}...dx_d\biggr)^{1/Q}.   $$
If $k=d,$

\begin{eqnarray}\label{harstmud}&\quad&
\biggl(\int\limits_{\mathbb R^d}
\left[\,\int\limits_{x_1-h_1}^{x_1+h_1}...\int\limits_{x_d-h_d}^{x_d+h_d}
F(u)\,du\right]^Q \,dx\biggr)^{1/Q}\nonumber\\
&\qquad&\le C(h_1...h_d)^{1/Q+1/q'} \biggl(\int\limits_{\mathbb
R^d} F^q(x)\,dx\biggr)^{1/q}.
\end{eqnarray}
\end{lemma}

Of course, the first $k$ variables are taken in (\ref{harstmu}) for
simplicity, the result is true for any $k$ variables.

\begin{proof} The proof is inductive. For $d=1,$ the result holds
true: (\ref{harst}). Supposing that it is true for $d-1,$
$d=2,3,...,$ let us prove (\ref{harstmu}) with $k=d.$ Applying
inductive assumption for the first $d-1$ variables, we obtain

\begin{eqnarray*}&\quad&\biggl(\int\limits_{\mathbb R^d}
\left[\,\int\limits_{x_1-h_1}^{x_1+h_1}...\int\limits_{x_d-h_d}^{x_d+h_d}
F(u_1,...,u_d)\,du_1...du_d\right]^Q \,dx_1...dx_d\biggr)^{1/Q}\\
&=&\biggl(\int\limits_{\mathbb R}\biggl\{\int\limits_{\mathbb
R^{d-1}}\left[\,\int\limits_{x_1-h_1}^{x_1+h_1}...\int\limits_{x_{d-1}-h_{d-1}}^{x_{d-1}+h_{d-1}}
\int\limits_{x_d-h_d}^{x_d+h_d}F(u_1,...,u_d)\,du_1...du_d\right]^Q\\
&\quad&\,dx_1...dx_{d-1}\biggr\}^{Q/Q}dx_d\biggr)^{1/Q}\\
&\le&C(h_1...h_{d-1})^{1/Q+1/q'}\biggl(\int\limits_{\mathbb R}
\biggl\{\int\limits_{\mathbb R^{d-1}} \left[\,
\int\limits_{x_d-h_d}^{x_d+h_d}F(x_1,...,x_{d-1},u_d)\,du_d\right]^q\\
&\quad&\,dx_1...dx_{d-1}
\biggr\}^{Q/q}dx_d\biggr)^{\frac{q}{Q}\frac{1}{q}}.\end{eqnarray*}
Applying now the generalized Minkowski inequality with exponent
$Q/q\ge1,$ we bound the right-hand side by, times a constant,

\begin{eqnarray*}(h_1...h_{d-1})^{1/Q+1/q'}\biggl(\int\limits_{\mathbb R^{d-1}}
\biggl\{\int\limits_{\mathbb R} \left[\,
\int\limits_{x_d-h_d}^{x_d+h_d}F(x_1,...,x_{d-1},u_d)\,du_d\right]^Q\\
dx_d \biggr\}^{q/Q}\,dx_1...dx_{d-1} \biggr)^{1/q}.
\end{eqnarray*}
To obtain (\ref{harstmu}), it remains again to make use of
(\ref{harst}) for the $d$-th variable.

If $k<d$, we just represent the considered integral as

\begin{eqnarray*}&\quad&\biggl(\int\limits_{\mathbb R^{d-k}}\biggl(
\int\limits_{\mathbb
R^k}\left[\,\int\limits_{x_1-h_1}^{x_1+h_1}...\int\limits_{x_k-h_k}^{x_k+h_k}
F(u_1,...,u_k,x_{k+1},...,x_d)\,du_1...du_k\right]^Q\\
&\quad&dx_1...dx_k\biggr)^{(1/Q)Q} \,dx_{k+1}...dx_d\biggr)^{1/Q}
\end{eqnarray*}
and apply the proved version to the inner integral. The proof is
complete. \hfill\end{proof}

We will also apply the following simple result.

\begin{lemma}\label{lem*}
Let $f\in C_0(\mathbb R^d)$, $D^{\bf{1}}f\in L_q(\mathbb R^d)$,
$1<q<\infty$, and partial derivatives $D^\eta f$ , $\eta\neq
\bf{1}$, are locally absolutely continuous on $(\mathbb R\setminus
\{0\})^d$ in each variable. Then

$$
\Vert \Delta_{h_1,\dots, h_d} f\Vert_\infty \le 2^{d/q'} (h_1\dots
h_d)^\frac{1}{q'}\Vert D^{\bf{1}}f\Vert_q.
$$
\end{lemma}

\begin{proof}
By H\"older's inequality,

\begin{equation*}
\begin{split}
\Vert \Delta_{h_1,\dots, h_d} f\Vert_\infty &\le
\int\limits_{x_1-h_1}^{x_1+h_1}\dots
\int\limits_{x_d-h_d}^{x_d+h_d}|D^\textbf{1}f(u_1,\dots,u_d)|\,du_1\dots
du_d\\
&\le \bigg(\int\limits_{x_1-h_1}^{x_1+h_1}\dots
\int\limits_{x_d-h_d}^{x_d+h_d}du_1\dots du_d\bigg)^\frac{1}{q'}
\Vert D^\textbf{1}f\Vert_q,
\end{split}
\end{equation*}
as required. \hfill\end{proof}

\bigskip

\section{One-dimensional result}\label{od}

Our main result in dimension one reads as follows (see \cite{Li}),
here we present a proof of the sufficiency different from that in
\cite{Li}.

\begin{theorem}\label{th1-1} Suppose a function $f$ satisfies
condition {\bf (N-1)}.

{\bf a)} Let $f(t)\in L_p(\mathbb R),$ $1\le p<\infty,$ and
$f'(t)\in L_q(\mathbb R),$ $1<q<\infty.$ If
$\frac{1}{p}+\frac{1}{q}>1,$ then $f\in A(\BR).$

{\bf b)} If $\frac{1}{p}+\frac{1}{q}<1,$ then there exists a
function $f$ satisfying {\bf (N-1)} such that $f(t)\in L_p(\mathbb
R)$ and $f'(t)\in L_q(\mathbb R)$ but $f\not\in A(\mathbb R).$
\end{theorem}

\begin{proof}

To prove {\bf b)} of the theorem, let us consider the function $m$
from the introduction. Suppose that $p\beta>1$ and
$q(\beta-\alpha+1)>1,$ with $\alpha\ne1.$ Simple calculations show
that $m\in L_p(\mathbb R)$ and $m'\in L_q(\mathbb R).$ If
$\frac{\beta}{\alpha}<\frac{1}{2},$ then $m\not\in A(\mathbb R).$
The last inequality is equivalent to $2\beta-\alpha+1<1.$ Therefore,

\begin{eqnarray*}\frac{1}{p}+\frac{1}{q}<2\beta-\alpha+1<1,\end{eqnarray*}
and the considered $m$ delivers the required counterexample.

\medskip

{\bf Proof of a)}. This is apparently the shortest possible proof.
Denoting

\begin{eqnarray}\label{dl}\Delta(h)=\left(\int\limits_{\mathbb R}
|\D_{h}f(t)|^2dt\right)^{1/2},                 \end{eqnarray} we
are going to prove the positive part by showing that

\begin{eqnarray}\label{os2}\sum\limits_{\nu=1}^\infty2^{-\nu/2}\Delta(h(-\nu))+
\sum\limits_{\nu=0}^\infty  2^{\nu/2}\Delta(h(\nu))<\infty.
\end{eqnarray}
It is obvious that for $h>0$

\begin{eqnarray}\label{pr3} \qquad|f(t+h)-f(t-h)|=|\int\limits_{t-h}^{t+h}
f'(s)\,ds|.              \end{eqnarray}

Let start with the first sum in (\ref{os2}) which is

\begin{eqnarray}\label{sum2-1}\sum\limits_{\nu=1}^\infty 2^{-\nu/2}
\biggl(\int\limits_{\mathbb R}|\D_{h(-\nu)}f(t)|^2dt\biggr)^{1/2}.
\end{eqnarray} Using (\ref{pr3}), we represent the integral as

\begin{eqnarray*}\biggl(\int\limits_{\mathbb R}|\D_{h(-\nu)}f(t)|
\left|\,\int\limits_{t-h(-\nu)}^{t+h(-\nu)}f'(s)\,ds\right|\,dt\biggr)^{1/2}.
\end{eqnarray*}
By H\"older's inequality, it is estimated via

\begin{eqnarray*}&\quad&\biggl(\int\limits_{\mathbb R}
|\D_{h(-\nu)}f(t)|^pdt\biggr)^{\frac{1}{2p}}\biggr(\,\int\limits_{\mathbb
R}\left[\,\int\limits_{t-h(-\nu)}^{t+h(-\nu)}|f'(s)|
\,ds\right]^{p'}\,dt\biggr)^{\frac{1}{2p'}}.       \end{eqnarray*}
Since $p'>q,$ we use (\ref{harst}) with $F(s)=|f'(s)|$ and $Q=p'.$
Therefore, the first sum in (\ref{os2}) is controlled by

\begin{eqnarray*}\|f\|_p^{1/2}\,\|f'\|_q^{1/2}\,\sum\limits_{\nu=1}^\infty
2^{-\frac{\nu}{2}(1-\frac{1}{p'}-\frac{1}{q'})},   \end{eqnarray*}
and is bounded since

\begin{eqnarray*}1-\frac{1}{p'}-\frac{1}{q'}=\frac{1}{p}+\frac{1}{q}-1>0.\end{eqnarray*}

\medskip

To handle the second sum, we represent it as (see (\ref{pr3}))

\begin{eqnarray}\label{ipl2}
&\quad&\biggl(\int\limits_{\mathbb R} |\D_{h(\nu)}f(t)|
\left|\,\int\limits_{t-h(\nu)}^{t+h(\nu)}f'(s)\,ds
\right|\,dt\biggr)^{1/2}.
\end{eqnarray}
Applying H\"older's inequality with the exponents $q'>1$ and $q$, we
estimate (\ref{ipl2}) via

\begin{eqnarray}&\quad&\label{sum1-2}\biggl(\int\limits_{\mathbb
R}|\D_{h(\nu)}f(t)|^{q'}dt\biggr)^{\frac{1}{2q'}}
\biggl(\int\limits_{\mathbb R}
\left[\int\limits_{t-h(\nu)}^{t+h(\nu)}|f'(s)|\,ds\right] ^{q}
dt\biggr)^{\frac{1}{2q}}.
\end{eqnarray}
By (\ref{pr3}) and Lemma \ref{lem*} in dimension one, the first
integral in (\ref{sum1-2}) is controlled by

$$\biggl(\int\limits_{\mathbb R}\,|\D_{h(\nu)}f(t)|^p\,
|\D_{h(\nu)}f(t)|^{q'-p}dt\biggr)^{\frac{1}{2q'}}\le
C\,h(\nu)^{\frac{q'-p}{q'}\frac{1}{2q'}}\,\|f\|_p^{\frac{p}{2q'}}\,
\|f'\|_q^{\frac{q'-p}{2q'}}.$$

To estimate the second one, we use (\ref{harst}) with $F(s)=|f'(s)|$
and $Q=q$. We thus estimate the second factor in (\ref{sum1-2}) via

\begin{eqnarray*} C[h(\nu)]^{1/2}\|f'\|_q^{1/2}.                      \end{eqnarray*}
Since $q'>p,$ the series

\begin{eqnarray*}\sum\limits_{\nu=1}^\infty 2^{\frac{q'-p}{q'}\frac{1}{2q'}}
\end{eqnarray*}
converges, which ensures the finiteness of (\ref{os2}).
\hfill\end{proof}

\bigskip

\section{Proofs of multidimensional results}\label{proofs}

We give, step by step, proofs of the results formulated in
Introduction.

\bigskip

\subsection{Proof of Theorem \ref{th1}.}

The proof is surprisingly very similar to that in dimension one.
When we deal with the part of the sum from Lemma C with

\begin{eqnarray*}
\sum_{k_1=1}^\infty \cdots \sum_{k_d=1}^\infty 2^ {-\frac 12
\sum_{j=1}^d k_j},                                 \end{eqnarray*}
we represent this sum as

\begin{eqnarray*} &\quad&
\sum_{k_1=1}^\infty \cdots \sum_{k_d=1}^\infty 2^ {-\frac 12
\sum_{j=1}^d k_j}\biggl(\int\limits_{\mathbb
R^d}|\D_{h(-k_1),\dots,h(-k_d)}f(x)|\\
&\times&\left|\,\int\limits_{x_1-h(-k_1)}^{x_1+h(-k_1)}...\int\limits_{x_d-h(-k_d)}^{x_d+h(-k_d)}
D^{\bf1}f(u)\,du\right|\,dx\biggr)^{1/2}
\end{eqnarray*}
and manage it exactly as in the either proof of the first sum in
dimension one.

Further, when we deal with the part of the sum from Lemma C with

\begin{eqnarray*}
\sum_{k_1=0}^\infty \cdots \sum_{k_d=0}^\infty 2^ {\frac 12
\sum_{j=1}^d k_j},
\end{eqnarray*} we proceed as in the (1st) proof of the second sum
in dimension one.

In both cases (\ref{harstmud}) from Lemma \ref{sh} is applied.

Finally, when we deal with the parts of the sum from Lemma C with

\begin{eqnarray*}
\sum_{i:\,\eta_i=0}2^ {-\frac 12 \sum_{j=1}^d k_j} \cdots
\sum_{i:\,\eta_i=1} 2^ {\frac 12 \sum_{j=1}^d k_j},
\end{eqnarray*}
where $\eta\ne{\bf0}$ and $\eta\ne{\bf1},$ the point is that we do
not need to treat the first sum at all: when the rest is bounded,
the series corresponding to $\eta_i=0$ converge automatically then.
As for the second sum, we proceed to it as in the proof of the
second sum in dimension one. Since we always apply Lemma \ref{sh}
with $Q=q,$ we get that (\ref{harstmu}) is reduced to usual
$L_{p_\eta}$ spaces.

The proof is complete. \hfill$\Box$

\medskip

\subsection{Proof of Theorem~\ref{th2}}

When we deal with the sum
\begin{equation}\label{ser2}
\sum_{k_1=1}^\infty\dots\sum_{k_d=1}^\infty 2^{-\frac
12(k_1+\dots+k_d)} \Vert
\Delta_{h({-k_1}),\dots,h({-k_d})}f\Vert_2
\end{equation}
we only need to use condition (\ref{condPo}). Choosing
$p_\eta^*>p_\eta$, ${\bf 0}\le \eta\le {\bf 1}$, such that
$$
\sum_{{\bf 1\le\eta\le 1}}\frac{1}{p_\eta^*}=2^{d-1}.
$$
Applying H\"older inequality and Lemma~\ref{sh}, we obtain
\begin{equation*}
\begin{split}
&\Vert \D_{h(-k_1),\dots, h(-k_d)} f\Vert_2\le C\bigg(\prod_{\bf
0\le\eta\le 1} \Vert \D_{h(-k_1),\dots, h(-k_d)}^\eta
f\Vert_{p^*_\eta}\bigg)^\frac{1}{2^d}\\
&\le C\bigg(\prod_{\bf 0\le\eta\le 1} \Vert
f\Vert_\infty^{1-\frac{p_\eta}{p^*_\eta}} (h(-k_1)^{\eta_1}\dots
h(-k_d)^{\eta_d})^{\frac{p_\eta}{p_\eta^*}}\Vert
D^{\bf{\eta}}f\Vert_{p_\eta}^\frac{p_\eta}{p_{\eta}^*}\bigg)^\frac{1}{2^d}\\
&=C \prod_{j=1}^d 2^{\frac{1}{2^d}(\sum_{{\bf 0\le\eta\le
1},\,\eta_j=1}\frac{p_\eta}{p_\eta^*})k_j}\bigg(\prod_{\bf
0\le\eta\le 1} \Vert f\Vert_\infty^{1-\frac{p_\eta}{p^*_\eta}}
\Vert
D^{\bf{\eta}}f\Vert_{p_\eta}^\frac{p_\eta}{p_{\eta}^*}\bigg)^\frac{1}{2^d}.
\end{split}
\end{equation*}
The last inequality together with the following inequality
$$
\sum_{{\bf 0\le\eta\le
1},\,\eta_j=1}\frac{p_\eta}{p_\eta^*}<2^{d-1}
$$
yield the convergence of the sum in (\ref{ser2}).

In what follows, we denote for simplicity $p_0=p_{\bf 0}$,
$p_d=p_{\bf 1}$.

Let us show that
\begin{equation}\label{ser1}
\sum_{k_1=0}^\infty\dots\sum_{k_d=0}^\infty 2^{\frac
12(k_1+\dots+k_d)} \Vert
\Delta_{h({k_1}),\dots,h({k_d})}f\Vert_2<\infty
\end{equation}

Assuming that condition (\ref{condOtr1}) holds with strong
inequality, we can choose $p_{0}^*>p_{0}$ such that

\begin{equation*}\label{ravp_0}
\frac{1}{p_{0}^*}+\sum_{\eta\neq {\bf 0}}\frac{1}{p_\eta}=2^{d-1}.
\end{equation*}

Applying then H\"older's inequality, Lemma~\ref{lem*} and
Lemma~\ref{sh}, we obtain

\begin{equation*}
\begin{split}
& \Vert \Delta_{h({k_1}),\dots,h({k_d})}f\Vert_2\\
&\le C \bigg(\Vert \Delta_{h({k_1}),\dots,h({k_d})}f
\Vert_\infty^{1-\frac{p_0}{p_0^*}}\Vert
f\Vert_{p_0}^{\frac{p_0}{p_0^*}} \prod_{\eta\neq \textbf{0}} \Vert
\Delta_{h{(k_1)},\dots,h{(k_d)}}^\eta f
\Vert_{p_{\eta}}\bigg)^\frac{1}{2^d}\\
&\le C
\bigg(2^{-(2^{d-1}+\frac{1}{p_d'}(1-\frac{p_0}{p_0^*}))(k_1+\dots+k_d)}
\Vert D^{\textbf{1}}f \Vert_{p_d}^{1-\frac{p_0}{p^*_0}}\Vert
f\Vert_{p_0}^{\frac{p_0}{p_0^*}} \prod_{\eta\neq \textbf{0}} \Vert
D^{\eta} f \Vert_{p_{\eta}}\bigg)^\frac{1}{2^d}.
\end{split}
\end{equation*}
The last inequality readily yields the convergence of
(\ref{ser1}).

If (\ref{condOtr1}) holds with equality, we can choose $p_0^*>p_0$
and $p_{e_1}^*>p_{e_1}$ such that
$$
\frac{1}{p_0^*}+\frac{1}{p_{e_1}^*}+\sum_{\eta\neq 0,\,\eta\neq
e_1}\frac{1}{p_\eta}=2^{d-1}.
$$
Note that we can choose $p_0^*$ to be sufficiently large and
$p_{e_1}^*$ to be sufficiently close to $p_{e_1}$.

Applying then H\"older's inequality, we obtain

\begin{equation}\label{eqProm1}
\begin{split}
\Vert \Delta_{h({k_1}),\dots,h({k_d})}f\Vert_2 \le C &\bigg(\Vert
\Delta_{h({k_1}),\dots,h({k_d})}f
\Vert_\infty^{2-\frac{p_0}{p_0^*}-\frac{p_{e_1}}{p^*_{e_1}}}\Vert
f\Vert_{p_0}^{\frac{p_0}{p_0^*}} \\
&\times\Vert \D^{e_1}_{h(k_1)} f
\Vert_{p_{e_1}}^{\frac{p_{e_1}}{p^*_{e_1}}} \prod_{\eta\neq
\textbf{0},\,\eta\neq e_1} \Vert
\Delta_{h{(k_1)},\dots,h{(k_d)}}^\eta f
\Vert_{p_{\eta}}\bigg)^\frac{1}{2^d}.
\end{split}
\end{equation}
From (\ref{eqProm1}), Lemma~\ref{sh} and Lemma~\ref{lem*} we get
\begin{equation*}
\begin{split}
&\Vert \Delta_{h({k_1}),\dots,h({k_d})}f\Vert_2\\
&=O(2^{-(\frac
12+\frac{1}{2^d}(-1+\frac{p_{e_1}}{p^*_{e_1}}+\frac{1}{p_d'}(2-\frac{p_0}{p_0^*}-\frac{p_{e_1}}{p^*_{e_1}})))k_1}
2^{-(\frac{1}{2}+\frac{1}{p_d'}(2-\frac{p_0}{p_0^*}-\frac{p_{e_1}}{p^*_{e_1}}))(k_2+\dots+k_d)}).
\end{split}
\end{equation*}
Thus, choosing $p_0^*$ and $p_{e_1}^*$ such that
$$
\frac{1}{p_d'}(2-\frac{p_0}{p_0^*}-\frac{p_{e_1}}{p^*_{e_1}})>1-\frac{p_{e_1}}{p^*_{e_1}},
$$
we obtain the convergence of (\ref{ser1}).

To complete the proof of the theorem, it remains to show the
convergence of the series of type

\begin{equation}\label{ser3}
\sum_{k_1=0}^\infty\dots\sum_{k_j=0}^\infty\sum_{l_{j+1}=1}^\infty\dots\sum_{l_d=1}^\infty
\frac{2^{\frac 12(k_1+\dots+k_j)}}{2^{\frac
12(l_{j+1}+\dots+l_d)}} \Vert \Delta_{h}f\Vert_2,
\end{equation}
where $1\le j\le d-1$ and
$h=(h({k_1}),\dots,h({k_j}),h({-l_{j+1}}),\dots,h({-l_{d}}))$.
Choosing $p_0^*>p_0$ and $p_{e_i}^*>p_{e_i}$, $i=j+1,\dots,d$ such
that

$$
\frac{1}{p_0^*}+\sum_{i=1}^j\frac{1}{p_{e_i}}+\sum_{i=j+1}^d\frac{1}{p_{e_i}^*}+\sum_{|\eta|>1}\frac{1}{p_\eta}=2^{d-1}.
$$
Applying H\"older's inequality, we obtain

\begin{equation}\label{eq1}
\Vert\Delta_{h}f\Vert_2\le C(S_1S_2S_3S_4)^\frac{1}{2^d},
\end{equation}
where
$$
S_1=\Vert \Delta_{h} f\Vert_{p_0^*},
$$
$$
S_2=\prod_{i=1}^j\Vert\Delta_{h({k_i})}^{e_i}f\Vert_{p_{e_i}},
$$
$$
S_3=\prod_{i=j+1}^d\Vert\Delta_{h({-l_i})}^{e_i}f\Vert_{p_{e_i}^*},
$$
and
$$
S_4=\prod_{|\eta|>1}\Vert \Delta_{h}^\eta f\Vert_{p_\eta}.
$$

Applying Lemma~\ref{lem*}, we get

\begin{equation}\label{eqS1}
\begin{split}
S_1&\le C \Vert f\Vert_{p_0}^\frac{p_0}{p_0^*}
\Vert\Delta_{h}f\Vert_\infty^{1-\frac{p_0}{p_0^*}} \le \Vert
f\Vert_{p_0}^\frac{p_0}{p_0^*} \Vert
f\Vert_\infty^{1-\frac{p_0}{p_0^*}-\e}
\Vert\Delta_{h}f\Vert_\infty^{\e}\\
&\le C 2^{-\frac{\e}{p_d'}(k_1+\dots+k_j-l_{j+1}-\dots-l_d)} \Vert
f\Vert_{p_0}^\frac{p_0}{p_0^*} \Vert
f\Vert_\infty^{1-\frac{p_0}{p_0^*}-\e} \Vert D^{\bf
1}f\Vert_{p_d}^\e,
\end{split}
\end{equation}
where $\e\in (0,1-\frac{p_0}{p_0^*})$.

Further, applying Lemma~\ref{sh}, we obtain

\begin{equation}\label{eqS2}
S_2\le C \prod_{i=1}^j 2^{-k_i}\Vert D^{e_i}f\Vert_{p_{e_i}},
\end{equation}
\begin{equation}\label{eqS3}
\begin{split}
S_3&\le C\prod_{i=j+1}^d \Vert
\Delta_{h({-l_i})}f\Vert_\infty^{1-\frac{p_{e_i}}{p^*_{e_i}}}
\Vert
\Delta_{h({-l_i})}f\Vert_{p_{e_i}}^{\frac{p_{e_i}}{p_{e_i}^*}}\\
&\le C \prod_{i=j+1}^d \Vert
f\Vert_\infty^{1-\frac{p_{e_i}}{p^*_{e_i}}}
2^{k_i{\frac{p_{e_i}}{p_{e_i}^*}}}\Vert
D^{e_i}f\Vert_{p_{e_i}}^{\frac{p_{e_i}}{p_{e_i}^*}},
\end{split}
\end{equation}
and
\begin{equation}\label{eqS4}
S_4\le C\prod_{|\eta|>1} 2^{\eta_1 k_1+\dots+\eta_j
k_j-\eta_{j+1}l_{j+1}-\dots-\eta_d l_d}\Vert D^\eta f
\Vert_{p_\eta}.
\end{equation}

Combining then (\ref{eq1}) and (\ref{eqS1})-(\ref{eqS4}), we arrive
at

\begin{equation*}
\begin{split}
\Vert \Delta_{h}f\Vert_2=O(\prod_{i=1}^j
2^{-(\frac12+\frac{\e}{2^d p_d'})k_i}\prod_{i=j+1}^d
2^{(\frac12+\frac{1}{2^d
}(\frac{\e}{p_d'}-1+\frac{p_{e_i}}{p_{e_i}^*}))l_i} ).
\end{split}
\end{equation*}
Choosing $\e\in (0,1-\frac{p_0}{p_0^*})$ such that
$$
\frac{\e}{p_d'}-1+\frac{p_{e_i}}{p_{e_i}^*}<0,\quad i=j+1,\dots,d,
$$
we obtain that (\ref{ser3}) is finite.

This completes the proof. \hfill$\Box$

\bigskip

\subsection{Proof of Corollary \ref{cor1}.}

Let us rewrite (\ref{coga}) as

\begin{eqnarray*}\sum\limits_{{\bf 0}\le\chi\le{\bf 1}}\gamma_\chi=d2^{d-1}+\epsilon.\end{eqnarray*}
For each $\chi,$ let us choose $p_\chi$ so that $\gamma_\chi
p_\chi=d+\frac{\epsilon}{2^d}.$ Then

\begin{eqnarray*}\sum\limits_{{\bf 0}\le\chi\le{\bf 1}}
\frac{d+\epsilon/2^d}{p_\chi}=d2^{d-1}+\epsilon.\end{eqnarray*}
Since

\begin{eqnarray*}\sum\limits_{{\bf 0}\le\chi\le{\bf 1}}\frac{\epsilon/2^d}{p_\chi}<\epsilon,\end{eqnarray*}
there holds

\begin{eqnarray*}\sum\limits_{{\bf 0}\le\chi\le{\bf 1}}\frac{d}{p_\chi}>d2^{d-1}.\end{eqnarray*}
This is equivalent to (\ref{condPoI*}), and hence $f\in A(\BR^d).$
\hfill$\Box$

\bigskip

{\bf Proof of Proposition~\ref{Prop1}.} Convergence of the series
like (\ref{ser2}) and (\ref{ser1}) is proved as in Theorems
\ref{th1} and \ref{th2}. Thus, in order to complete the proof of
the proposition, it suffices to demonstrate the convergence of
series of type (\ref{ser3}). We will restrict ourselves to the
case $j=d-1$, that is, to the series

\begin{equation}\label{ser4}
\sum_{k_1=0}^\infty\dots\sum_{k_{d-1}=0}^\infty\sum_{l_d=1}^\infty
\frac{2^{\frac 12(k_1+\dots+k_j)}}{2^{\frac 12 l_d}} \Vert
\Delta_{h({k_1}),\dots,h({k_d}),h({-l_{d}})}f\Vert_2.
\end{equation}
For $j<d-1,$ the proof goes along the same lines as in the following
arguments.

Let
\begin{equation}\label{eq.prop1}
\sum_{|\eta|\ge\frac d2}\frac{1}{p_\eta}\ge 2^{d-1},
\end{equation}
otherwise the proof is obvious.

Assume that there is a collection $\{p_\eta^*\}$ such that
$p_\eta^*>p_\eta$ when $|\eta|\le\frac d2-1$,
$p_{\eta^{(1)}}^*>p_{\eta^{(1)}}$, where
$\eta^{(1)}=(\eta_1^{(1)},\dots,\eta_{d-1}^{(1)},1)$,
$|\eta^{(1)}|=\frac d2$, and

\begin{equation}\label{eq.prop2}
\sum_{|\eta|\le\frac
d2-1}\frac{1}{p_\eta^*}+\frac{1}{p_{\eta^{(1)}}^*}+\sum_{|\eta|\ge\frac
d2,\,\eta\neq \eta^{(1)}}\frac{1}{p_\eta}=2^{d-1}.
\end{equation}
Observe that $p_\eta^*$ can be chosen arbitrary large when
$|\eta|\le\frac d2-1.$

For convenience, we set $p_\eta^*=p_\eta$ when $|\eta|\ge\frac d2$
è $\eta\neq \eta^{(1)}$, while $\eta_1^{(1)}=\dots=\eta_{\frac
d2-1}^{(1)}=\eta_d^{(1)}=1$ and $\eta_{\frac
d2}^{(1)}=\dots=\eta_{d-1}^{(1)}=0$. Let also
$h=(h(k_1),\dots,h(k_{d-1}),h(-l_d))$.

Applying H\"older's inequality, we obtain

\begin{equation}\label{eq.prop3}
\Vert \Delta_h f\Vert_2\le C(S_1 S_2 S_3 S_4)^\frac{1}{2^d},
\end{equation}
where
$$
S_1=\Vert \Delta_h f\Vert_{p_0^*},
$$
$$
S_2=\prod_{|\eta|=1}\Vert \Delta_h f\Vert_{p_{\eta}^*},
$$
$$
S_3=\Vert \Delta_h^{\eta^{(1)}} f\Vert_{p_{\eta}^*},
$$
and
$$
S_4=\prod_{|\eta|>1,\,\eta\neq \eta^{(1)}}\Vert \Delta_h^\eta
f\Vert_{p_{\eta}^*}.
$$

As is shown above (see (\ref{eqS1}))

\begin{equation}\label{eq.propS1}
S_1=O\left(2^{-\frac{\e}{p_d'}(k_1+\dots+k_{d-1}-l_d)}\right).
\end{equation}
Further, applying Lemma~\ref{sh}, we get

\begin{equation}\label{eq.propS2}
\begin{split}
S_2 &\le C\Vert
\D_h^{\eta^{(1)}-e_d}f\Vert_{p_{\eta^{(1)}-e_d}^*}^{\frac
d2}\prod_{j=\frac d2}^{d-1}\Vert \D_h^{e_j}f\Vert_{p_{e_j}^*}\\
&\le C 2^{-\frac{d}{2}(k_1+\dots+k_{\frac d2-1})}\Vert
D^{\eta^{(1)}-e_d} f\Vert^{\frac
d2}_{p_{\eta^{(1)}-e_d}^*}\prod_{j=\frac d2}^{d-1}2^{-k_j}\Vert
D^{e_j}f \Vert_{p_{e_j}^*},
\end{split}
\end{equation}

\begin{equation}\label{eq.propS3}
\begin{split}
S_3 &\le C \Vert \D_h^{\eta^{(1)}}
f\Vert_\infty^{1-\frac{p_{\eta^{(1)}}}{p^*_{\eta^{(1)}}}} \Vert
\D_h^{\eta^{(1)}}
f\Vert_{p_{\eta^{(1)}}}^{\frac{p_{\eta^{(1)}}}{p^*_{\eta^{(1)}}}}\\
&\le C \Vert
f\Vert_\infty^{1-\frac{p_{\eta^{(1)}}}{p^*_{\eta^{(1)}}}}\bigg(
\prod_{j=1}^{\frac d2-1}2^{-k_j}2^{l_d}\Vert D^{\eta^{(1)}} f
\Vert_{p^*_{\eta^{(1)}}}
  \bigg)^{\frac{p_{\eta^{(1)}}}{p^*_{\eta^{(1)}}}}
\end{split}
\end{equation}
and
\begin{equation}\label{eq.propS4}
\begin{split}
S_4=O\bigg(\prod_{j=1}^{\frac d2-1} 2^{-(2^{d-1}-2)k_j}
\prod_{j=\frac d2}^{d-1} 2^{-(2^{d-1}-1)k_j}
2^{(2^{d-1}-2)l_d}\bigg)
\end{split}
\end{equation}

Combining (\ref{eq.prop3}) and (\ref{eq.propS1})-(\ref{eq.propS4}),
we obtain

\begin{equation*}
\begin{split}
\Vert \Delta_h f\Vert_2=O\bigg( \prod_{j=1}^{\frac d2-1}
&2^{-(\frac12+\frac{1}{2^d}(\frac d2+
\frac{p_{\eta^{(1)}}}{p^*_{\eta^{(1)}}} +
\frac{\e}{p_d'}-1))k_j}\\
&\times\prod_{j=\frac d2}^{d-1} 2^{-(\frac12+\frac{\e}{2^d
p_d'})k_j} 2^{( \frac12+\frac{1}{2^d}(
\frac{p_{\eta^{(1)}}}{p^*_{\eta^{(1)}}}   -2+\frac{\e}{p_d'})
)l_d}\bigg).
\end{split}
\end{equation*}
Hence, choosing $\e$ to be small enough, we readily get the
convergence of the series in question.

Now, if (\ref{eq.prop2}) holds for no collection
$\{p^*_\eta\}_{|\eta|\le\frac d2-1}$ and $p^*_{\eta^{(1)}}$, we
suppose that there is a collection $\{p^*_\eta\}_{|\eta|\le\frac
d2-1}$, $p^*_{\eta^{(1)}}$ and $p^*_{\eta^{(2)}}$ such that
$p_\eta^*>p_\eta$ for $|\eta|\le\frac d2-1$,
$p_{\eta^{(j)}}^*>p_{\eta^{(j)}}$, where
$\eta^{(j)}=(\eta_1^{(j)},\dots,\eta_{d-1}^{(j)},1)$,
$|\eta^{(j)}|=\frac d2$, $j=1,2$, and

\begin{equation*}
\sum_{|\eta|\le\frac
d2-1}\frac{1}{p_\eta^*}+\frac{1}{p_{\eta^{(1)}}^*}+\frac{1}{p_{\eta^{(2)}}^*}+\sum_{|\eta|\ge\frac
d2,\,\eta\neq \eta^{(1)},\eta^{(2)}}\frac{1}{p_\eta}=2^{d-1}.
\end{equation*}
Note that we can now choose $p^*_{\eta^{(1)}}$ and $p_\eta^*$
arbitrary large when $|\eta|\le\frac d2-1$.

We then repeat the above way of reasoning replacing $\Vert
\D_h^{\eta^{(1)}}\Vert_{p^*_{\eta^{(1)}}}$ with $\Vert
\D_h^{{\eta^{(1)}-e_d}} f \Vert_{p^*_{\eta^{(1)}-e_d}}$. This is
always possible, since the numbers ${p^*_{\eta^{(1)}}}$ and
${p^*_{\eta^{(1)}-e_d}}$ can be chosen arbitrary large.

The proof can be completed then by repeating this procedure the
needed number of times. \hfill$\Box$

\end{document}